\documentclass[12pt]{article}
\usepackage{amsthm}
\usepackage{amssymb}
\usepackage{amsfonts, mathrsfs}
\usepackage{amsmath}
\usepackage{algorithm,enumerate}
\def\G{\Gamma}
\newcommand{\pr}{\mathbf{Pr}}

\newcommand{\set}[1]{\left\{ #1 \right\}}

\usepackage{tikz}
\usetikzlibrary{arrows}
\usetikzlibrary{decorations}
\usetikzlibrary{shapes.misc}
\newcommand{\pic}[1]{
\begin{tikzpicture}
#1
\end{tikzpicture}}

\newcommand{\rdown}[1]{\left\lfloor#1\right\rfloor}

 \newcommand{\bag}{\begin{align}}
\newcommand{\bags}{\begin{align*}}
\newcommand{\eag}{\end{align*}}
\newcommand{\eags}{\end{align*}}

\newtheorem{thm}{Theorem}
\newtheorem{lem}[thm]{Lemma}

\newtheorem{notn}[thm]{Notation}

\newcommand\cB{\mathcal B}

\newcommand\cE{\mathcal E}

\newcommand\f{\phi}

\newcommand\bd{\mathbf{d}}
\newcommand\ex{\mathbb{E}}

\newcommand\bfrac[2]{\left(\frac{#1}{#2}\right)}
\newcommand{\leo}{\leq_O}
\newcommand{\geo}{\geq_O}
\newcommand{\beq}[2]{\begin{equation}\label{#1}#2\end{equation}}
\newcommand{\brac}[1]{\left(#1\right)}

\begin{document}
\title{Finding perfect matchings in random cubic graphs in linear expected time}
\author{ Michael Anastos and Alan Frieze\thanks{Research supported in part by NSF Grant DMS1363136}}

\maketitle
\begin{abstract}
In a seminal paper on finding large matchings in sparse random graphs, Karp and Sipser \cite{KS} proposed two algorithms for this task. The second algorithm has been intensely studied, but due to technical difficulties, the first algorithm has received less attention. Empirical results in \cite{KS} suggest that the first algorithm is superior. In this paper we show that this is indeed the case, at least for random cubic graphs. We show that w.h.p. the first algorithm will find a matching of size $n/2 - O(\log n)$ on a random cubic graph (indeed on a random graph with degrees in $\set{3,4}$). We also show that the algorithm can be adapted to find a perfect matching w.h.p. in $O(n)$ time, as opposed to $O(n^{3/2})$ time for the worst-case.
\end{abstract}
\section{Introduction}
Given a graph $G=(V,E)$, a matching $M$ of $G$ is a subset of edges such that no vertex is incident to two edges in $M$. Finding a maximum cardinality matching is a central problem in algorithmic graph theory. The most efficient algorithm for general graphs is that given by Micali and Vazirani  \cite{MV} and runs in $O(|E||V|^{1/2})$ time.

In a seminal paper, Karp and Sipser \cite{KS} introduced two simple greedy algorithms for finding a large matching in the random graph $G_{n,m},m=cn/2$ for some positive constant $c>0$. Let us call them  Algorithms 1 and 2 as they are in \cite{KS}. Algorithm 2 is simpler than Algorithm 1 and has been intensely studied: see for example Aronson, Frieze and Pittel \cite{AFP}, Bohman and Frieze \cite{BF}, Balister and Gerke \cite{BG} or Bordenave and Lelarge \cite{BL}. In particular, \cite{AFP} together with Frieze and Pittel \cite{FP} shows that w.h.p. Algorithm 2 finds a matching that is within $\tilde{\Theta}(n^{1/5})$  of the optimum, when applied to $G_{n,m}$. Subsequently, Chebolu, Frieze and Melsted \cite{CFP} showed how to use Algorithm 2 as a basis for a linear expected time algorithm, when $c$ is sufficiently large.

Algorithm 2 proceeds as follows (a formal definition of Algorithm 1 is given in the next section). While there are isolated vertices  it deletes them. After which, while there are vertices of degree one in the graph, it chooses one at random and  adds the edge incident with it to the matching and deletes the endpoints of the edge. Otherwise, if the current graph has minimum degree at least two, then it adds a random edge to the matching and deletes the endpoints of the edge. 

In the same paper Karp and Sipser proposed another algorithm for finding a matching that also runs in linear time. This was Algorithm 1. The algorithm sequentially reduces the graph until it reaches the empty graph. Then it unwinds some of the actions that it has taken and grows a matching which is then output. Even though it was shown empirically to outperform Algorithm 2, it has not been rigorously analyzed. In this paper we analyze Algorithm 2 in the special case where the graph is random with a fixed degree sequence $3\leq d(i)\leq 4$ for $i=1,2,\ldots,n$. We prove the following:
\begin{thm}\label{main}
Let $G$ be a random graph with degree sequence $3\leq d(i)\leq 4$ for $i=1,2,\ldots,n$. Then 
\begin{enumerate}[(a)]
\item Algorithm 1 finds a matching of size $n/2-O(\log n)$, w.h.p.
\item Algorithm 1 can be modified to find a (near) perfect matching in $O(n)$ time w.h.p. and in expectation.
\end{enumerate}
\end{thm}
A (near) perfect matching is one of size $\rdown{n/2}$. Note that in the case of cubic graphs, it is known that they have (near) perfect matchings w.h.p., see Bollob\'as \cite{Bol}. Note also that it was shown by Frieze, Radcliffe and Suen \cite{FRS} that w.h.p. Algorithm 2 finds a matching of size $n/2-\tilde{\Theta}(n^{1/5})$. \footnote{Recently, the junior author has extended Theorem \ref{main} to random $r$-regular graphs for all $3\leq r=O(1)$.}
\section{The Algorithm}
The algorithm that is given in \cite{KS}  can be split into two parts. The first part sequentially reduces the graph until it reaches the empty graph. Then the second part reverses part of this reduction and grows a matching which is then output. 

To reduce the graph, 
\begin{enumerate}[(1)]
\item First, while there are vertices of degree 0 or degree 1 the algorithm removes them along with any edge incident to them. The edges removed at this stage will be part of the output matching. 
\item Second, while there are vertices of degree 2 the algorithm contracts them along with their two neighbors. That is the induced path $(x,y,z)$ is replaced by a single contracted vertex $y_c$ whose neighbors are those of $x,z$ other than $y$. The description in \cite{KS} does not explicitly say what to do with loops or multiple edges created by this process. In any case, such creations are very rare. We say a little more on this in Section \ref{details}.

In the unwinding, if we have so far constructed a matching containing an edge $\set{y_c,\xi}$ incident with $y_c$ and $\xi$ is a neighbor of $x$ then in our matching we replace this edge by $\set{x,\xi}$ and $\set{y,z}$. If there is no matching edge so far chosen incident with $y_c$ then we add an arbitrary one of $\set{x,y}$ or $\set{y,z}$ to our matching.
\item Finally if the graph has minimum degree 3 then a random vertex is chosen among those of maximum degree and then a random edge incident to that vertex is deleted. These edges will not be used in the unwinding.
\end{enumerate}   
\subsection{Idea of proof:} No mistakes are made while handling vertices of degree 0,1 or 2. Each mistake decreases the size of the final matching produced by one from the maximum size. We will show that mistakes occur only at parts of the graph that have become denser than is likely. 

We show that w.h.p. the maximum degree remains $O(\log^{2}\nu)$ where $\nu$ is the number of vertices remaining and so as long as $\nu\log n$, say, then w.h.p. there will be no dense subgraphs and the algorithm will not make any mistakes. This explains the $O(\log n)$ error term. Finally, we assert that removing an edge incident to a vertex of a maximum degree will help to control the maximum degree, explaining this choice of edge to delete.

\subsection{Details}\label{details}
The precise algorithm that we analyze is called {\sc reduce-construct} The algorithm description given in \cite{KS} is not explicit  in how to deal with loops and multiple edges, as they arise. We will remove loops immediately, but keep the multiple edges until removed by other operations. 

We assume that our input (multi-)graph $G = G([n],E)$ has degree sequence $\bd$ and is generated by the configuration model of Bollob\'as \cite{Bol}. Let $W=[2\nu]$, $2\nu=\sum_{i=1}^n d(i)$, be our set of {\em configuration points} and let $\Phi$ be the set of {\em configurations} i.e. functions $\phi:W \mapsto [n]$ that such that $|\f^{-1}(i)|=d(i)$ for every $i \in [n]$. Given $\phi \in \Phi$ we define the graph $G_\phi=([n],E_\phi)$ where $E_\phi=\{\{\phi(2j-1),\phi(2j)\}: j\in [\nu] \}$. Choosing a function $\phi \in \Phi$ uniformly at random yields a random (multi-)graph $G_\phi$ with degree sequence $\bd$. 

It is known that conditional on $G_\f$ being simple, i.e. having no loops or multiple edges, it is equally likely to be any graph that has degree sequence $\bd$. Also, if the maximum degree is $O(1)$ then the probability that $G_\f$ is simple is bounded below by a positive quantity that is independent of $n$. Thus results on this model can be translated immediately to random simple graphs.

We split the {\sc reduce-construct} Algorithm into the {\sc reduce} and {\sc construct} algorithms which we present separately.

\textbf{Algorithm} \textsc{Reduce}:

The input $G_0=G_\f$ where we condition on there being no loops.\\ 
$i=\hat{\tau}=0$.
\\ \textbf{While} $G_i=(V_i,E_i) \neq (\emptyset,\emptyset)$ do: 
\begin{itemize}
\item[]\textbf{If} $\delta(G_i)=0$: Perform a {\bf vertex-0 removal}: choose a  vertex of degree 0 and remove it from $V_i$.
\item[]\textbf{Else if} $\delta(G_i)=1$: Perform a {\bf vertex-1 removal}: choose a random vertex $v$ of degree 1 and remove it along with its neighbor $w$ and any edge incident to  $w$. 
\item[]\textbf{Else if} $\delta(G_i)=2$: Perform a {\bf contraction}: choose a random vertex $v$ of degree 2. Then replace $\{v\} \cup N(v)$ ($v$ and its neighbors $N(v)$) by a single vertex $v_c$. For $u\in V \setminus (\{v\} \cup N(v))$, $u$ is joined to $v_c$ by as many edges as there are in $G_i$ from $u$ to $\{v\} \cup N(v)$.  Remove any loops created.
\item[]\textbf{Else if } $\delta(G_i)\geq 3$: Perform a {\bf max-edge removal}: choose a random vertex of maximum degree and remove a random edge incident with it.
\\ \textbf{End if}
\item[]\textbf{If} the last action was a max-edge removal, say the removal of edge $\{u,v\}$ and in the current graph we have $d(u)=2$ and $u$ is joined to a single vertex $w$ by a pair of parallel edges then perform an {\bf auto correction contraction}: contract $u,v$ and $w$ into a single vertex. Remove any loops created.
\\ \textbf{End If}
\item[] Set $i=i+1$ and let $G_i$ be the current graph.
\end{itemize}
\textbf{End While}
\\Set $\hat{\tau}=i$.

Observe that we only reveal edges (pairs of the form $(\phi(2j-1),\phi(2j)): j\in [\nu]$) of $G_\phi$ as the need arises in the algorithm. Moreover the algorithm removes any edges that are revealed. Thus if  we let $\bd(i)$  be the degree sequence of $G_i$ then, given $\bd(i)$ and the actions performed by {\sc reduce} until it generates $G_i$ we have that $G_i$ is uniformly distributed among all configurations with degree sequence $\bd(i)$ and no loops.

Call a contraction that is performed by {\sc reduce} and involves only 2 vertices \emph{bad} i.e. one where we contract $u,v$ to a single vertex given that $G$ contains a parallel pair of the edge $\set{u,v}$ and $u$ has degree 2. Otherwise call it \emph{good}. Observe that a bad contraction can potentially be a mistake while a good contraction is never a mistake. By introducing the auto correction contraction we replace the bad contraction of the vertex set $\{u,w\}$, as presented in the description of {\sc reduce}, with the good contraction of the vertex set $\{v,u,w\}$. Note that we do not claim that all bad contractions can be dealt with in this way. We only show later that other instance of bad contractions are very unlikely.

We now describe how we unwind the operations of {\sc reduce} to provide us with a matching.

\textbf{Algorithm} {\sc construct}:

\textbf{Input:} $G_0,G_1,...,G_{\hat{\tau}}$ - 
the graph sequence produced by {\sc reduce},
an integer $j\in \{0,1,...,\hat{\tau}\}$ and  a matching $M_{j}$ of $G_j$. (We allow the possibility of stopping {\sc reduce} before it has finished.
If we do so when $|V(G_j)|=\Theta(n^{2/3})$ then, given that $G_j$ has a perfect matching w.h.p., we can use the $O(|E||V|^{1/2})$ algorithm of \cite{MV} applied to $G_j$ to find a perfect matching $M_j$ of $G_j$ in $O(n)$ time. Thereafter we can use {\sc construct} to extend $M_j$ to a matching of $G_0$.)
\\ \textbf{For $i=1$ to $j$ do}: 
\begin{itemize}
\item[]\textbf{If} $\delta(G_{j-i})=0$: Set $M_{j-i}=M_{j-i+1}$
\item[]\textbf{Else if} $\delta(G_{j-i})=1$: Let $v$ be the vertex of degree 1 chosen at the $(j-i)th$ step of {\sc reduce} and let $e$ be the edge that is incident to $v$ in $G_{j-i}$. Then, 
Set $M_{j-i}=M_{j-i+1}\cup\{e\}$.
\item[]\textbf{Else if} $\delta(G_{j-i})=2$: Let $v$ be the vertex of degree 2 selected in $G_{j-i}$. If $|N(v)|=1$ i.e. $v$ is joined to a single vertex by a double edge in $G_{j-i}$, set $M_{j-i}=M_{j-i+1}$. Else let $N(v)=\{u,w\}$ and $v_c$ be the new vertex resulting from the contraction of $\{v,u,w\}$. If $v_c$ is not covered by $M_{j-i+1}$ then set $M_{j-i}=M_{j-i+1}\cup\{v,u\}$. Otherwise assume that $\{v_c,z\}\in M_{j-i+1}$ for some $z\in V(G_{j-i})$. Without loss of generality assume that in $G_{j-i}$, $z$ is connected to $u$. Set
$M_{j-i}=(M_{j-i+1}\cup \{\{v,w\},\{u,z\}\})\setminus \{ v_c,z\}$.
\item[]\textbf{Else if }: $\delta(G_{j-i})\geq 3$: Set $M_{j-i}=M_{j-i+1}$
\end{itemize}
\textbf{End For}

For a graph $G$ and $j\in \{0,1,...,\hat{\tau}\}$ denote by $R_0(G,j)$ and $R_{2b}(G,j)$ the number of times that  {\sc reduce} has performed a vertex-0 removal and a bad contraction respectively until it generates  $G_j$. For a graph $G$ and a matching $M$ denote by $\kappa(G,M)$ the number of vertices that are not covered by $M$. The following Lemma determines the quality of the output of the {\sc reduce-construct} algorithm.
\begin{lem}\label{correction}
Let $G$ be a graph and $M$ be the output of the Reduce-Backtrack algorithm applied to $G$. Then, for $j\geq 0$,
\beq{true}{
\kappa(G,M)=R_0(G,j)+R_{2b}(G,j)+\kappa(G_j,M_j).
}
\end{lem}
\begin{proof}
Let $G=G_0,G_1,...,G_{\hat{\tau}}$ be the sequence of graphs produced by {\sc reduce} and let $M_j,M_{j-1},...,M_0=M$ be the sequence of matchings produced by {\sc construct}. Let $R_0(G,j,i)$ and $R_{2b}(G,j,i)$ be the number of vertex-0 removals and bad contractions performed by {\sc reduce} going from $G_{j-i}$ to $G_j$. We will prove that for every  $0\leq i\leq j$, 
\beq{troo}{
\kappa(G_{j-i},M_{j-i})=R_0(G,j,i)+R_{2b}(G,j,i)+\kappa(G_j,M_j).
} 
Taking $i=j$ yields the desired result.

For $i=0$, equation \eqref{true} holds as $R_0(G,j,0)=R_{2b}(G,j,0)=0$. Assume inductively that \eqref{true} holds for $i=k-1$ where $k$ satisfies $0<k\leq j$. For $i=k$, if a max-edge deletion was performed on $G_{j-k}$ then $|V_{j-k}|=|V_{j-k+1}|$. Furthermore, $R_0(G,j,k)= R_0(G,j,k-1)$ and $R_{2b}(G,j,k)=R_{2b}(G,j,k-1)$ and hence \eqref{true} continues to hold. If a vertex-0 deletion or a bad contraction was performed on $G_{j-k}$ then $|V_{j-k}|=|V_{j-k+1}|+1$ and $M_{j-k}=M_{j-k+1}$. In the case of a vertex-0 deletion we have $R_0(G,j,k)= R_0(G,j,k-1)+1$ and $R_{2b}(G,j,k)=R_{2b}(G,j,k-1)$ and both sides of \eqref{troo} increase by one. In the case of a bad contraction we have $R_0(G,j,k)= R_0(G,j,k-1)$ and $R_{2b}(G,j,k)=R_{2b}(G,j,k-1)+1$ and again  both sides of \eqref{troo} increase by one. Finally if a good contraction or a vertex-1 removal was performed on $G_{j-k}$ then $R_0(G,j,k)= R_0(G,j,k-1)$ and $R_{2b}(G,j,k)=R_{2b}(G,j,k-1)$. At the same time we have that $\kappa(G_{j-i},M_{j-i})=\kappa(G_{j-i+1},M_{j-i+1})$, completing the induction.
\end{proof}
\subsection{Organizing the actions taken by {\sc reduce}}
We do not analyze the effects of each action taken by {\sc reduce} individually. Instead we group together sequences of actions, into what we call {\em hyperactions}, and we analyze the effects of the individual hyperactions.


We construct a sub-sequence $\G_0=G,\G_1,\ldots,\G_\tau$ of $G_0,G_1,...,G_{\hat{\tau}}$. Every hyperaction, starts with a max-edge removal and it consists of all the actions taken until the next max-edge removal. We let $\Gamma_{i}$ be the graph that results from performing the first $i$ hyperactions. Thus $\Gamma_{i}$ is the $i$th graph in the sequence $G_0,G_1,...,G_{\hat{\tau}}$ that has minimum degree at least 3 and going from $\Gamma_{i}$ to $\Gamma_{i+1}$ {\sc reduce} performs a max-edge removal followed by a sequence of vertex-0, vertex-1 removals and contractions. Thus $\Gamma_0,\Gamma_1,...,\Gamma_{\tau}$ consists of all the graphs in the sequence $G_0,G_1,...,G_{\hat{\tau}}$ with minimum degree at least 3.  
\subsection{Excess and Proof Mechanics}\label{list}
The central quantity of this paper is the \emph{excess} which we denote by $ex_\ell(\cdot)$. For a graph $G$ and a positive  integer $\ell$ we let 
$$ ex_\ell(G):=\sum_{v \in V(G)} [d(v)-\ell]\mathbb{I}(d(v)>\ell).$$
{\bf Hyperactions of Interest}\\
For the analysis of {\sc reduce} we consider 7 distinct hyperactions (sequences of actions) which we call hyperactions of Type 1,2,3,4,5,33 and 34 respectively. In the case that the maximum degree is larger than 3 we consider the following hyperactions: we have put some diagrams of these hyperactions at the end of the paper.
\begin{itemize}
\item[]\textbf{ Type 1}: A single max-edge removal,
\item[]\textbf{ Type 2}: A max edge-removal followed by an auto correction contraction.
\item[]\textbf{ Type 3}: A single max-edge removal followed by a good contraction. 
\item[]\textbf{ Type 4}: A single max-edge removal followed by 2 good contractions.
In this case we add the restriction that there are exactly 6 distinct vertices $v,u,x_1,x_2,w_1,w_2$ involved in this hyperaction and they satisfy the following:
(i) $v$ is a vertex of maximum degree, it is adjacent to $u$ and $\{u,v\}$ is removed during the max-edge removal, (ii) $d(u)=d(x_1)=d(x_2)=3$, (iii) $N(u)=\{v,x_1,x_2\}$, $N(x_1)=\{u,x_2,w_1\}$ and $N(x_2)=\{u,x_1,w_2\}$. (Thus $\set{u,x_1,x_2}$ form a triangle.) The two contractions have the same effect as contracting $\{u,x_1,x_2,w_1,w_2\}$ into a single vertex.
\end{itemize}
 In the case that the maximum degree equals 3 we also consider the following hyperactions:
\begin{itemize}
\item[] \textbf{Type 5}:  A max-edge removal followed by 2 good contractions that interact.
In this case the 5 vertices $u,v,x_1,x_2,z$ involved in the hyperaction satisfy the following:
(i) $\{u,v\}$ is the edge removed by the max-edge removal, (ii) $N(v)=\{u,x_1,x_2\}$, $N(u)=\{v,x_1,z\}$, (so $\set{u,v,x_1}$ form a triangle), (iii) $|(N(x_1) \cup N(x_2) \cup N(z)) \setminus \set{u,v,x_1,x_2,z}|\geq 3$. This hyperaction has the same effect as contracting all of $\set{u,v,x_1,x_2,z}$ into a single vertex.
\item[] \textbf{Type 33}: A max-edge removal followed by 2 good contractions that do not interact. There are 6 distinct vertices involved $v,v_1,v_2,u,u_1,u_2$. During the max-edge removal $\{u,v\}$ is removed. Thereafter each of the 2 sets of vertices $\{v,v_1,v_2\}$ and $\{u,u_1,u_2\}$ is contracted to a single vertex.
\item[] \textbf{Type 34}: A max-edge removal followed by 3 good contractions. There are 8 distinct vertices involved $v,v_1,v_2,v,u,u_1,u_2,w_1,w_2$. During the max-edge removal $\{u,v\}$ is removed. The conditions satisfied by $v,u,u_1,u_2,w_1,w_2$ and the actions that are performed on them are similar to the ones in a hyperaction of Type 4. The difference now is that $v$ has degree 3 before the hyperaction. In addition $\{v,v_1,v_2\}$ is contracted into a single vertex.
\end{itemize}
We divide Hyperactions of Type 3 into three classes. Assume that during a Hyperaction of Type 3 the set $\{v,a,b\}$ is contracted, $v$ is the contracted vertex and $v_c$ is the new vertex. We say that such a Hyperaction is of \textbf{Type 3a} if $d(v_c)=d(a)+d(b)-2$, is of \textbf{Type 3b} if $d(v_c)=d(a)+d(b)-4$ and is of \textbf{Type 3c} if $d(v_c)<d(a)+d(b)-4$. Note that in general, $d(v_c)=d(a)+d(b)-2-2\eta_{a,b}$, where $\eta_{a,b}$ is the number of edges joining $a,b$ plus the number of edges joining $\set{u,v}$ after the max-edge removal.

With the exception of a Hyperaction of Type 3c, where $\eta_{a,b}\geq 2$, we refer to the Hyperactions of interest as good Hyperactions. We call any Hyperaction that is not good, including a Hyperaction of Type 3c, bad. 

We next state three lemmas, whose proofs will be deferred to Section \ref{LemmasDegree}.
\begin{lem}\label{hyper}
Let $i\geq 0$ and assume that $\G_i$  satisfies 
$ex_\ell(\G_i)\leq \log^2 |V(\G_i)|$ for some $3\leq \ell =O(1)$.
 Then with probability $1-o(|V(\G_i)|^{-1.9})$ the hyperaction that {\sc reduce} applies to $\G_i$ is good. In addition, it only applies a Hyperaction of Type 2, 3b or 4 with probability $o(|V(\G_i)|^{-0.9})$.
\end{lem}
The fact that w.h.p. we perform one hyperaction out of a small set of possibilities helps in controlling $ex_\ell$. Namely we have  
\begin{lem}\label{drift}
Suppose that $\Gamma_i$ satisfies $ex_4(\G)\leq \log^2 |V(\G_i)|$.
Then either $ex_4(\G_i)=0$ or 
$$\ex(ex_4(\Gamma_{i+1})-ex_4(\G_i)|\G_i)\leq -\frac{1}{3}.$$
Moreover with probability $1-o(|V(\G_i)|^{-1.9})$ we have that 
$$|ex_4(\G_i)-ex_4(\G_{i+1})|\leq 2.$$
\end{lem}
\begin{lem}\label{4inf}
Suppose that $\omega\geq \log n$. Then for every $i\in \{0,1,..,\tau\}$, $|V(\Gamma_i)|\geq \omega$ implies that with probability $1-O(\omega^{-0.8})$, we have $ex_4(\Gamma_{i})\leq \log^2 |V(\G_i)|$. Furthermore with probability $1-O(\omega^{-0.8})$ there exists $\G \in \{\Gamma_{0},\Gamma_{1},..,\Gamma_{\tau}\}$ that satisfies $\omega\leq |V(\G)|\leq 2\omega$ and $ex_4(\G)=0$.
\end{lem}
Lemma  \ref{hyper} states that as long as $ex_\ell$ stays small then w.h.p. we perform only specific hyperactions. Lemmas \ref{drift} and \ref{4inf} imply that as long we perform only those hyperactions then w.h.p. $ex_4$ stays small. Furthermore, as none of those hyperactions consist of a vertex-0 removal or a bad contraction making we can appeal to Lemma \ref{correction}. We use the above three lemmas to prove the following:
\begin{lem}\label{reduction}
Let $\Gamma_h$ be as promised by Lemma \ref{4inf} that is, $\G_h \in \{\Gamma_{0},\Gamma_{1},..,\Gamma_{\tau}\}$ and it satisfies $\log n\leq \omega \leq V(\G_h)\leq 2\omega$ and $ex_4(\G)=0$. Then with probability $1-O(\omega^{-0.8})$, for any matching $M_h$ of $\Gamma_h$, the matching $M$ generated by {\sc construct} with $j=h$, $M_h$ as an input satisfies
\begin{equation}\label{kappa}
\kappa(G,M)=\kappa(\Gamma_h,M_h).
\end{equation} 
\end{lem}
\begin{proof}
For $i<h$, Lemma \ref{4inf} states that w.h.p. $ex_4(\Gamma_{i})\leq \log^2 |V(\G_i)|$. Lemma \ref{hyper} then implies that for $i<h$ with probability 
$$1-\sum_{i=1}^{h}  o(|V(\G_i)|^{-1.9})\geq 1-\sum_{i=1}^{h}  o((2|E(\G_i)|/3)^{-1.9}) \geq 1- \sum_{i=\omega}^{\infty}  o(2i^{-1.9}/3) =1-o(\omega^{-0.9})$$
only good hyperactions take place. For the first inequality we used the fact that $\G_i$ has minimum degree 3. For the second we used the fact that any Hyperaction decreases the number of edges, hence $|E(\G_i)|$ is decreasing with respect to $i$.

Hence w.h.p., until {\sc reduce} generates $\Gamma_h$ it performs no vertex-0 removals or bad contractions.
Let $\hat{h} \in \{0,1,,...,\hat{\tau}\}$ be such that $G_{\hat{h}}=\G_h$.
Then from the above $R_0(G,\hat{h})=R_{2b}(0,\hat{h})=0$.
 Finally in this case, Lemma \ref{correction} implies that 
$$\kappa(\G,M_h)=\kappa(G_{\hat{h}},M_h)=\kappa(\Gamma_h,M_h).$$
\end{proof}
Observe that by  a straightforward application of Lemma \ref{reduction} with $\omega=\log n$ and $M_h$ being the matching produced by {\sc construct} in the case $j=\widehat{\tau}$ we have that the matching $M$ output by the {\sc reduce-construct} algorithm satisfies
 $$\kappa(G,M)=\kappa(\Gamma_h,h) \leq |V(\Gamma_h)| \leq 2\log n.$$
This proves Part (a) of Theorem \ref{main}.
\subsection{Finding a perfect matching in linear time} 
Lemma \ref{reduction} suggests that when we run {\sc construct} we should finish our algorithm by finding a maximum matching $M'$ of $\Gamma_h$. The final ingredient for the proof of Theorem \ref{main}(b) is given by the next Theorem which implies that $\Gamma_h$ has a parfect matching w.h.p.
\\
\begin{thm}\label{34match}
Let $\bd$ be a degree sequence where $3\leq d(i)\leq 4$ for all $i$. Let $G$ be a random graph with degree sequence $\bd$ and no loops. Then $G$ has a (near)-perfect matching with probability $1-O(n^{-3/4})$. 
\end{thm}
We defer the proof of Theorem \ref{34match} to Section \ref{Proofs Matchings}. To prove Part (b) of Theorem \ref{main}, we run algorithm {\em Reduce} until we find $\Gamma_h$ as promised in Lemma \ref{4inf}. Here we take $\omega=n^{2/3}$. We know from Lemma \ref{reduction} and Lemma \ref{34match} that $\Gamma_h$ has a (near)-perfect matching with probability $1-O(\omega^{-0.8})-O(\omega^{-3/4})=1-O(n^{-1/2})$ that can be expanded to a (near)-perfect matching of the original graph $G$. Moreover since $|V(\G_h)|\leq 2\omega$, such a matching can be found in $O(\omega^{3/2})=O(n)$ time. The rest of the algorithm can be executed in $O(n)$ time and this completes the high probability part of the proof of Part (b). Also, if {\sc reduce-construct} fails, then we can resort to an $O(n^{3/2})$ time algorithm. This only happens with probability $O(n^{-1/2})$ and this yields the claim in Part (b) about the expected running time. 
\section{Proofs of Lemmas \ref{hyper}, \ref{drift} and \ref{4inf}}\label{LemmasDegree}
\begin{notn}
We sometimes write $A\leo B$ in place of $A=O(B)$ for aesthetic purposes.
\end{notn}
\begin{notn}
Let $K$ be an arbitrary positive integer and $b\in \{0,1\}$. For a random graph $G$ and $v\in V(G)$, let $\cB_K(G,v,b)$ be the event that $G$ spans a subgraph that contains $v$, spans $a\leq K$ vertices and $a+\ell$ edges.
\end{notn}
We show that {\sc reduce} either performs one of the good hyperactions given in Section \ref{list} or $\cB_K(G,v,1)$ occurs where $v$ is the vertex of maximum degree chosen by {\sc reduce}. We will also use the following standard notation: For a graph $G$, $j,\ell \in \mathbb{N}$ we let:
\begin{itemize}
\item $n(G):=|V(G)|$, $e(G):=|E(G)$,
\vspace{-2mm} 
\item $\delta(G)$ and $\Delta(G)$ be the minimum and maximum degree of $G$ respectively, 
\vspace{-2mm}
\item $n_j(G)$ be the number of vertices of $G$ of degree $j$,
\vspace{-2mm}
\item $p_j(G):=\frac{jn_j(G)}{2|E(G)|}$ and $p_{>j}(G):=\sum_{h>j} p_h(G)$,
\vspace{-2mm}
\item $ ex_\ell(G):=\sum_{v \in V(G)} [d(v)-\ell]\mathbb{I}(d(v)>\ell)$, 
\end{itemize}
We denote by $n_i,e_i, n_{j,i},p_{j,i},p_{>j,i}$ and $ex_{\ell,i}$ the corresponding quantities in relation to $\G_i$.  
\begin{lem}\label{dence}
Let $K$ be an arbitrary fixed positive integer. Let $\bd$ be a degree sequence of length $n$ that satisfies $ex_\ell(G)\leq  \log^2 n$ for some $3\leq \ell=O(1)$. Let $G$ be a random graph with degree sequence $\bd$ and no loops. Let $b \in \{0,1\}$, then $\pr(\cB_K(G,v,b))=o(n^{-0.9-b})$.
\end{lem}
\begin{proof}
Let $G$ be a random graph with degree sequence $\bd$. The fact that $ex_\ell(G)\leq   \log^2 n$ implies that $G$ has no loops with probability bounded below by a positive constant (see for example \cite{FK}). Hence the condition of having no loops can be ignored in the proof that events have probability $o(1)$. Also it implies that $\Delta=\Delta(G)\leq \ell+ ex_\ell(G) \leq \ell+\log^2 n$.

Let $2m=\sum_{i=1}^n d(i)\leq \ell n+ex_{\ell}(G) =\Theta(n)$. Then for vertex $v$ and for $b=0,1$ the probability that $G$ spans a subgraph that covers $v$, spans $a\leq K$ vertices and $a+b$ edges can be bounded above by 
\begin{align*}
&\leo \sum_{a=2}^{K} \binom{n}{a-1} (\Delta a)^{2(a+b)} \frac{(2m-2(a+b))!}{2^{m-(a+b)}[m-(a+b)]!}\times \frac{2^{m}m!}{(2m)!}\\ 
&\leo \sum_{a=2}^{K} n^{a-1} \Delta^{2(a+b)} \frac{m(m-1)...(m-(a+b)+1)}{2m(2m-1)...(2m-2(a+b)+1)}\\ 
&\leo \sum_{a=2}^{K}  n^{a-1}  \Delta^{2(a+b)}  m^{-(a+b)}\\
&=o(n^{-0.9-b}).
\end{align*}
\end{proof}
We will drop the subscript $K$ from $\cB$. Taking $K=20$ will easily suffice for the rest of the proof.  Thus $\cB(G,v,b)=\cB_{20}(G,v,b)$.
\subsection{Proof of Lemma \ref{hyper}}
Let $v$ be the vertex of maximum degree chosen by {\sc reduce} and let $u$ be the vertex adjacent to $v$ such that $\{u,v\}$ is chosen for removal. We will show that if $\cB(\G_i,v,1)$ does not occur then {\sc reduce} performs one of the hyperactions given in Section \ref{list}.
Also observe that if a Hyperaction of Type 2,3b,4,5 or 34 occurs then $\cB(\G_i,v,0)$ occurs. Lemma \ref{dence} states that $\pr(\cB(\G_i,v,0))=o(|V(\G_i)|^{-0.9})$ thus proving the second part of Lemma \ref{hyper}. Also note that if  a Hyperaction of Type 3c, occurs, corresponding to a bad Hyperaction, then $\cB(\G_i,v,1)$ occurs. Lemma \ref{dence} states that $\pr(\cB(\G_i,v,1))=o(|V(\G_i)|^{-1.9})$.

{\bf Case A: $d(v) \geq 4$.}\\
If $d(u)\geq 4$ then a  hyperaction of Type 1 is performed.  
Thus assume  $d(u)=3$ and consider the cases where $|N(u)|=1,2,3$, (recall that we allow parallel edges but not self-loops). 

{\textbf{Case A1: $|N(u)|=1$}}.\\
$u$ is connected to $v$ by 3 parallel edges and so $\cB(\G_i,v,1)$ occurs.	

{\textbf{Case A2: $|N(u)|=2$}}.\\
Let $N(u)=\{v,u'\}$ and $S=\set{u,u',v}$. Let $T=(N(u')\cup N(v) )\setminus S$. If $|T|\leq 2$ then we have $d(S)\geq 10$ and either $S$ spans more than 3 edges or $S \cup T$ spans at least 7 edges. In both cases  $\cB(\G_i,v,1)$ occurs. Assume then that  $|T|\geq 3$. Now exactly one of $\set{u,u'}$, $\set{u,v}$ is repeated, else $\cB(\G_i,v,1)$ occurs.  If $\set{u,u'}$ is repeated then we perform an auto correction contraction resulting to a Hyperaction of Type 2 . If $\set{u,v}$ is repeated then we contract the remaining path $(u',u,v)$. Hence we have performed a Hyperaction of Type 3b.

{\textbf{Case A3: $|N(u)|=3$.} }\\
Let $N(u)=\set{v,x_1,x_2}$ and $T=(N(x_1)\cup N(x_2))\setminus\{u,x_1,x_2\}$. 

{\textbf{Sub-case A3a: $|T|\leq 1$}. \\
$\{u,x_1,x_2\} cup T$ spans at least $(d(u)-1+d(x_1)+d(x_2)+|T|)/2 \geq 4+|T|/2$ edges and the event $\cB(\G_i,v,1)$ occurs.

{\textbf{Sub-case A3b: $T=\{w_1,w_2\}$}}. \\
If  $v$ is at distance less than 6 from $\{u\} \cup N(u)$ in $\G\setminus \{ u,v\}$ then  $\cB(\G_i,v,1)$ occurs. To see this consider the subgraph $H$ spanned by $\{u,v,x_1,x_2,w_1,w_2,y\}$ and the vertices on the shortest path $P$ from $v$ to $u$ in $\G_i\setminus \{ u,v\}$. Here $y$ is the neighbor of $v$ on $P$. It must contain at least two distinct cycles. One contained in $\set{u,x_1,x_2,w_1,w_2}$ and $u,v,P$. If there is no edge from $x_1$ to $x_2$ then $\{v,u,x_1,x_2,w_1,w_2\}$ spans at least 7 edges and so $\cB(\G_i,v,1)$ occurs.

Thus we may assume that $N(x_1)=\{u,x_2,w_1\}$, $N(x_2)=\{u,x_1,w_2\}$ and $v\notin \{w_1,w_2\}\cup N(w_1) \cup N(w_2)$. We may also assume that $\set{w_1,w_2}$ is not an edge of $\G$, for otherwise $\{u,x_1,x_2,w_1,w_2\}$ contains two distinct cycles and $\cB(G_i,v,1)$ occurs. The algorithm {\sc Reduce} proceeds by contracting $u,x_1,x_2$ into a single vertex $x'$.  $x'$ has degree 2 and the algorithm proceeds by performing a contraction of $x',w_1,w_2$ into a new vertex $w'$. Let $S=N\{w_1,w_2\}\setminus \{x_1,x_2\}$. If $|S|\leq 3$ then $\cB(\G_i,u,1)$ occurs. To see this observe that $w_1,w_2$ must then have a common neighbor $w_3$ say. Consider the subgraph $H$ spanned by $\{u,x_1,x_2,w_1,w_2,w_3\}$. $H$ contains at least 7 edges and 6 vertices. If $|S|\geq 4$ then the new vertex has degree 4 and the sequence of actions taken by {\sc reduce} corresponds to a hyperaction of Type 4.
\vspace{3mm}
\\{\textbf{Sub-case A3c: $|T| \geq 3$.}}
After the removal of $\{v,u\}$ we contract $\{u,x_1,x_2\}$ into a single vertex of degree at least 3, hence a hyperaction of Type 3 is performed.
\vspace{3mm}
\\
{\bf Case B: $d(v)=d(u)=3$.}\\
Let $\G'=\G_i\setminus \{e\}$ where $e=\{v,u\}$.
\vspace{3mm}
\\{\textbf{Case B1: In $\G'$, $u$ and $v$ are at distance at least 4.}}\\ 
If $|N(N (u))|$ and $|N(N(v))| \leq 3$ then $\cB(\G_i,v,1)$ occurs. Thus we can assume that either $|N(N (u))|=4$ and/or $|N(N(v))|=4$. If both $|N(N (u))|,|N(N(v))|=4$ then {\em Reduce} will perform 2 good contractions and this amounts to a hyperaction of Type 33. Assume then that $|N(N (u))|=4$ and that $|N(N (v))|\leq 3$. If $|N(N (v))|=3$ then again {\em Reduce} will perform 2 good contractions amounting to a hyperaction of Type 33. If $|N(N (v))|=2$ and so $v$ is in a triangle then {\em Reduce} will perform a hyperaction of Type 34. Finally, if $|N(N (v))|=1$ then $\cB(\G_i,v,1)$ occurs.

{\textbf{Case B2: In $\G'$, $u$ and $v$ are at distance 3.}}\\
In $\G_i$ there is a cycle $C$ of length 4 containing $u,v$. If in $\G'$ we find that $|N(N (u))|\leq 3$ or $ |N(N(v))| \leq 3$ or $|N(u)\cap N(N(v))|>1$ or $|N(v)\cap N(N(u))|>1$ then $\cB(\G_i,v,1)$ occurs. This is because the graph spanned by $\{u,v\} \cup N(u)\cup N(v) \cup N(N(u) \cup N(N(v))$ in $\G_i$ will contain a cycle distinct from $C$. Assume this is not the case. Then after the max-edge removal of $\{u,v\}$ we have a contraction of $\{u\} \cup N(u)$ followed by a contraction of $\{u\} \cup N(u)$. Observe that neither contraction reduces the size of $N(N(u))$ or $N(N(v))$. Thus {\sc reduce} performs a hyperaction of Type 33.
\vspace{3mm}
\\{\textbf{Case B3: In $\G'$, $u$ and $v$ are at distance 2.}\\
In the case that $u,v$ have 2 common neighbors in $\G'$ we see that $\cB(\G_i,v,1)$ occurs. Assume then that they have a single common neighbor $x_1$. Let $z,x_2$ be the other neighbors of $u,v$ respectively. Then either $\cB(\G_i,v,1)$ occurs or {\sc reduce} performs a hyperaction of Type 5.
\vspace{3mm}
\\{\textbf{Case B4: In $\G'$, $u$ and $v$ are at distance 1}.\\
So here we have that $\set{u,v}$ is a double edge in $\G_i$. Let $x,y$ be the other neighbors of $u,v$ repectively in $\G$. Assuming that $\cB(\G_i,v,1)$ does not occur, {\sc reduce} performs a max-edge removal followed by a single good contraction and this will be equivalent to a hyperaction of Type 3, involving the contraction of one of $x,u,v$ or $u,v,y$.
\qed
\subsection{Proof of Lemma \ref{drift}}
The inequality $ex_{4,i} \leq \log^2 n_i$ implies that $p_{j,i}\leq \log^2n_i/2e_i =o(n_i^{-0.95})$ for $5\leq j \leq \Delta$. It also implies that the maximum degree $\Delta$ of $\G$ satisfies $\Delta=O(\log^2 n)$. 

In the case that $\cB(G,v,1)$ occurs, i.e. with probability $O(n_i^{-1.9})$ (see Lemma \ref{hyper}) 
$$|ex_{4,i}-ex_{4,i+1}|\leq 2e_i \leq 4n_i+ex_{4,i} \leq 5n_i.$$ 
Observe that if an action of Type 2,3b,4,5 or 34 takes place then $v$ lies on a subgraph with $a<12$ vertices and $a$ edges. Lemma \ref{hyper} states that this occurs with probability $o(n^{-0.9})$. For all the above hyperactions we have $|ex_{4,i+1}-ex_{4,i}|\leq 2$. This follows from the fact that performing a contraction can increase $ex_4$ by at most 2. This is because the initial vertices with degrees say $2,d_1,d_2$ contributed $\max\{ 0, (d_1-4)+(d_2-4) \}$ to $ex_{4,i}$ while the new contracted vertex has degree $d_1+d_2-2$ and contributes $\max\{ 0, d_1+d_2-2-4) \}$ to $ex_{4,i+1}$. Moreover observe that if a hyperaction of Type 5, 33 or 34 occurs then all the vertices involved have degree 3. If $\cB(\G,v,1)$ does not occur then a hyperaction of Type 5 will increase $ex_{4,i}$ by 1 (since there will be one new vertex of degree 5). Hyperactions of Type 33 or Type 34 do not change $ex_4$. Thus it remains to examine the effects of a hyperaction of Type 1 or of Type 3a. 

If  $ex_{4,i}=0$ then a hyperaction of Type 1 does not increase $ex_{4,i}$ while a hyperaction of Type 3 could increase it by 2. 

If $ex_4(\G)>0$ then the $i^{th}$ hyperaction starts with a max-edge removal. \\
{\bf Case 1:} If the smaller degree vertex involved is of degree larger than 3, then this results in a hyperaction of Type 1. This happens with probability $(1+o(1))(1-p_{3,i})$. Furthermore in this case $ex_{4,i}-2\leq ex_{4,i+1}\leq ex_{4,i}-1$.
 (The $(1+o(1))$ factor arises because of $O(1)$ degree changes during the hyperaction.)\\
{\bf Case 2:} If the smaller degree vertex $v$ involved is of degree 3 then a contraction is performed. And this occurs with probability $(1+o(1))p_{3,i}$. If the contraction results in a vertex of degree at least 3 then we have a hyperaction of Type 3a, and not of Type 1. Also this is the only way for a hyperaction of Type 3a to occur. Let the other two vertices involved in the contraction be $a,b$ and  have degrees $d_a,d_b$ respectively. Now $d_a=d_b=3$ with probability $(1+o(1))p_{3,i}^2$, resulting in a new vertex that has degree at most 4. In this case, $ex_{4,i+1}-ex_{4,i}= -1$ (the -1 here originates from the max-edge removal). Else if $d_a=3,d_b=4$ then we have $-1\leq ex_{4,i}-ex_{4,i}\leq -1+1=0$. And this occurs with probability $(1+o(1))2p_{3,i}p_{4,i}$. Else if $d_a=d_b=4$ then we have $-1\leq ex_{4,i+1}-ex_{4,i}\leq -1+2= 1$ (this occurs with probability $p_{3,i}p_{4,i}^2$. Otherwise a vertex of degree at least 5 is involved and given our upper bound on $ex_{4,i}$, this happens with probability $o(1)$. If the event $\cB(\G,v,1)$ does not occur then the new contracted vertex has degree   $d_a+d_b-2$. Hence $-1\leq ex_{4,i} - ex_{4,i} \leq -1+2=1.$ 

Thus in all cases $|ex_{4,i} - ex_{4,i}|\leq 2$ and if $ex_{4,i}>0$ then
\begin{align*}
\ex[ex_{4,i+1}-ex_{4,i}| \G_i] &\leq -(5n_i)\cdot o(n_i^{-1.9})+ 2n_i^{-0.95} - (1-p_{3,i}) -p_{3,i}^3 + p_{3,i}p_{4,i}^2+o(1)
\\&\leq - (1-p_{3,i}) -p_{3,i}^3 +p_{3,i}(1-p_{3,i})^2+o(1)
 \leq -\frac13
\end{align*} 
(The final expression in $p_{3,i}$ is maximized at $p_{3,i}=1/2$, for $p_{3,i}\in [0,1]$.)
\qed
\subsection{Proof of Lemma \ref{4inf}} 
We start by proving the following Lemma:
\begin{lem}\label{intervals}
Let $\G_h$ be such that $ex_{4,h}=0$. Then with probability $1-o(n_j^{-1.8})$ there exists $1\leq j\leq 0.25 \log^2n_h$ satisfying $ex_{4,h+j}=0$. Furthermore
$ex_{4,h+i}\leq \log^2n_{h+i}$ for $i\leq j$.  
\end{lem}
\begin{proof}
If $ex_{4,h+1}=0$ then we are done. Otherwise Lemma \ref{drift} implies that 
$ex_{4,h+1}\in\{1, 2\}$ with probability $1-o(n_h^{-1.9})$.

Let $\cE_T$ be the event that for $j\leq 0.25\log^2 n_h$ {\sc reduce} performs only hyperactions of Type 1,2,3,4,5,33 or 34. Such a hyperaction reduces the vertex set by at most 8. Lemma \ref{hyper} implies that $\cE_T$ occurs with probability $1-o(n^{-1.8}_h)$.  Moreover if $\cE_T$ occurs for $j<0.25\log^2 n_h$ then $|n_{h+j}| \geq n_h-8\cdot 0.25 \log^2 n_h$. In addition from Lemma \ref{drift} we have that with probability $1-o(n^{-1.8})$, $|ex_{4,h+j}-ex_{4,h+j-1}|<2$ for $j<0.25 \log^2 n$ hence $ex_{4,h+j} \leq 2\cdot 0.25 \log^2 n_h \leq \log^2 n_{h+j}$. Finally conditioned on $\cE_T$ since $ex_{4,h+1}=1$ or 2, the probability that there is no $2\leq j\leq 0.25\log^2 n_h$ satisfying $ex_{4,h+j}=0$ is bounded by the probability that the sum of $0.25 \log^2 n-2$ independent random variables with magnitude at most 2 and expected value smaller than $-1/3$ (see Lemma \ref{drift}) is positive. From Hoeffding's inequality \cite{Hoeff} we have that the  later probability is bounded by $\exp\set{-\frac{2(\frac13\log^2n_h-3)^2}{\log^2 n_h}}=o(n^{-2}_h)$.
\end{proof}
Now let $\Gamma_0,\Gamma_{i_1},...,\Gamma_{i_\ell}$ be the subsequence of $\Gamma_{0},\Gamma_{1},..,\Gamma_{\tau}$ that includes all the graphs that have $ex_4=0$ and at least $\omega$ vertices. Then since $\G_i$ has minimum degree 3 and $e_i$ is decreasing with respect to $i$ using  Lemma \ref{intervals} we have that with probability 
$$1-\sum_{i:n_i\geq \omega}o( n_i^{-1.8} )\geq 1-\sum_{i:n_i\geq \omega} 2e_i^{-1.8}/3 \geq  1-\sum_{i=\omega}^{\infty} i^{-1.8}=1-o(\omega^{-0.8})$$
we have that for $j<\ell$  we have $n_{i_j} -n_{i_{j+1}} \leq 8\cdot 0.25\log ^2 n_{i_j}=2\log ^2 n_{i_j}$ and all the graphs $\Gamma_i$ preceding $\Gamma_{i_\ell}$ in $\Gamma_0,\Gamma_1,...,\Gamma_\tau$ satisfy $ex_{4,i} \leq \log^2 n_i$.  Suppose now that $n_{i_\ell}> 2\omega$. Then the above argument implies that w.h.p. there is $j>i_\ell$ such that $ex_{4,j}=0$ and $n_j\geq n_{i_\ell}-2\log^2n_{i_\ell}\geq \omega$ and this contradicts the definition of $i_\ell$. Thus, w.h.p., $\omega\leq n_{i_\ell}\leq 2\omega$ and hhis completes the proof of Lemma \ref{4inf}.
\qed
\section{Existence of a Perfect Matching}\label{Proofs Matchings}
We devote this section to the proof of  Lemma \ref{34match}.  As discussed in the previous section it is enough to prove that given a degree sequence $\bd=(d(1),...,d(n))$ consisting only of 3's and 4's, if we let $G$ be a random configuration (multi)-graph with degree sequence $\bd$, then w.h.p. $G$ has a (near)-perfect matching (i.e we can lift the condition that $G$ has no loops). We will first assume that $n$ is even and verify Tutte's condition. That is for every $W\subset V$ the number of odd components induced by $V \setminus W$, $q(V/W)$, is not larger that $|W|$. We split the verification of Tutte's condition into two lemmas. 
\begin{lem}\label{minimal}
Let $W\subset V$ be a set of minimum size that satisfies $q(V/W)>|W|$. Then with probability $1-O(n^{-3/4})$, $|W| > 10^{-5}n$.  
\end{lem}
\begin{lem}\label{maximal}
Let $W\subset V$ be a set of maximum size that satisfies $q(V/W)>|W|$. Then  with probability $1-O(n^{-3/4})$, $|W| < 10^{-5}n$.  
\end{lem}
Lemmas \ref{minimal} and \ref{maximal} together imply Tutte's condition as there exists no set with size that is both strictly larger and strictly smaller than $10^{-5}n$. In the proof of these lemmas we use the following estimates.
\begin{lem}\label{estimates}
The number of distinct partitions of a set of size $2r$ into 2-element subsets, denoted by $\phi(2r)$, satisfies $\phi(2r)=\Theta\left(\bfrac{2r}{e}^r\right)$. Also for $\ell<r$ we have $\phi(2r)\leq 2r^{r-\ell} \bfrac{2\ell}{e}^\ell$.
\end{lem}
\begin{proof}
To generate a matching we first choose a permutation of the $2r$ items and then we pair the $(2i-1)$th item with the $2i$th item. Therefore, using Stirling's approximation we have 
$$\phi(2r)=\frac{{(2r)}!}{2^r r!} =\frac{\Theta(\sqrt{2r} \bfrac{2r}{e}^{2r})}{\Theta(2^r \sqrt{r} \bfrac{r}{e}^r)}= \Theta\bigg( \bfrac{2r}{e}^r\bigg).$$
Also
$$\phi(2r)=\frac{{(2r)}!}{2^r r!} \leq \frac{(2r)^{r-\ell}{(2\ell)}!}{2^{r-\ell} 2^\ell \ell!}   \leq 2r^{r-\ell} \bfrac{2\ell}{e}^\ell,$$
where we have used $(2r)!\leq (2r)^{2(r-\ell)}(2\ell)!$.
\end{proof}
\subsection{Proof of Lemma \ref{minimal}:}
Let $W$ be a  set satisfying $q(V \setminus W)>|W|$ of minimum size and assume $2\leq w=|W|\leq 10^{-5}n$. We can rule out the case $w=1$ from the fact that with probability $1-O(1/n)$, $G$ will be 3-connected, see e.g. the proof of Theorem 10.8 in \cite{FK} . Let $C_z$ be a component spanned by $V\setminus W$ of maximum size and let $r=|C_z|$. 

\textbf{Case 1: $r=|C_z|\leq 0.997n$.}} In this case we can partition $V\setminus W$ into two parts $V_1,V_2$ such that (i) each $V_l,l=1,2$ is the union of components of $V\setminus W$, (ii) $|V_1|\geq |V_2|$, and (iii) $|V_2|\geq (n-(r+w))/2\geq 10^{-3}n$. 

Let $d_2=d(V_2)$ and $d_W=d(W)$. Out of the $d_W$ endpoints in $W$ (i.e. configuration points that correspond to vertices in $W$), $\ell\leq d_W$ are matched with endpoints in $V_2$ and the rest with endpoints in $V_1$. 

For fixed $i,w,d_W$ the probability that there are sets $V_1,V_2,W$ with $w=|W|, d(W)=d_W$ and $|V_2|=i$ satisfying $1\leq w\leq 10^{-5}n, 10^{-3}n\leq i\leq 0.5n$ and $d_W\leq 4w\leq 0.04i$, such that $V_1 \times V_2$ spans no edges is bounded by 
\begin{align*}
p_1 &\leq  \sum_{\ell=0}^{d_W}
\binom{n}{i} \binom{n-i}{w} \binom{d_W}{\ell} \frac{\phi(d_2+\ell)\cdot \phi(2m-d_2-\ell) }{\phi(2m)}\\ 
&\leo \sum_{\ell=0}^{d_W}
\bfrac{en}{i}^i \bfrac{en}{w}^w 2^{d_W}  \frac{\bfrac{d_2+\ell}{e}^{(d_2+\ell)/2} 
\bfrac{2m-d_2-\ell}{e}^{(2m-d_2-\ell)/2} }{\bfrac{2m}{e}^m}\\ 
&\leq\sum_{\ell=0}^{d_W}
 \bfrac{en}{i}^i \bfrac{100en}{i}^{i/100} 2^{i/25}  \bfrac{d_2+\ell}{2m}^{(d_2+\ell)/2} 
\bigg(1-\frac{d_2+\ell}{2m}\bigg)^{(2m-d_2-\ell)/2} \\  
& \leo \sum_{\ell=0}^{d_W}
  \bfrac{1600(en)^{101}}{i^{101}}^{i/100} \bfrac{d_2+\ell}{2m}^{(d_2+\ell)/2} \exp\set{- \frac{d_2+\ell}{2}\brac{1-\frac{d_2+\ell}{2m}}} 
\end{align*}
For the third line we used the fact that $w\leq i/100$ and $d_W\leq 4w \leq i/25$. 

Let $f(x)=x^xe^{-x(1-x)}$ and $L(x)=\log f(x)$. Then $L''(x)=x^{-1}+2$ and so $L$ and hence $f$ is convex for $x>0$. Now $d_2+\ell \in J= [3i,4.04i]$ and since $\bfrac{d_2+\ell}{2m}^{(d_2+\ell)/2}\exp\set{- \frac{d_2+\ell}{2}\brac{1-\frac{d_2+\ell}{2m}}}=f\bfrac{d_2+\ell}{2m}^m$ we see that its  maxima are at the endpoints of $J$. In general $3i\leq 3n/2 \leq m$. However when $d_2+\ell=4.04i$ we have that 
\beq{2m}{
2m\geq 4.04i+3(n-i-w)\geq 4.04i+ 3(n-1.01i)= 3n+1.01i.
}
{\bf Case 1a: $d_2+\ell=3i$.}\\
We have $\frac{d_2+\ell}{2m} \leq \frac{3i}{3n}\leq \frac{1}{2}$ and $(d_2+\ell)(1-\frac{d_2+\ell}{2m})\geq 3i/2$.  Therefore,
\begin{align*}
p_1&\leo w   \bfrac{1600(en)^{101}}{i^{101}}^{i/100}\bfrac{i}{n}^{3i/2} e^{-3i/4} \\
&= w  \bigg[\frac{1600e^{26}}{2^{49}}  \bfrac{2i}{n}^{49}\bigg]^{i/100} \\ 
&\leq w \bigg[ e^{-1/2}  \bfrac{2i}{n}^{49}\bigg]^{i/100}\\
&\leq w e^{-i/200}.
\end{align*}
{\bf Case 1b: $d_2+\ell=4.04i$.}\\
It follows from \eqref{2m} that $\frac{d_2+\ell}{2m} \leq \frac{4.04i}{3n+1.01i} \leq 0.577$ where the second inequality uses $i\leq n/2$. It follows from this that $(d_2+\ell)(1-\frac{d_2+\ell}{2m})/2\geq 0.85i$. Hence,
\begin{align*}
p_1&\leo w   \bfrac{1600(en)^{101}}{i^{101}}^{i/100} \bfrac{4.04i}{3n+1.01i}^{2.02i} e^{-0.85i}\\
&\leo w  \bigg[1600e^{16} \bigg(\frac{n}{i}\bigg)^{101} \bigg( \frac{4.04i}{3n}\bigg)^{101} \bigg(\frac{4.04i}{3n+1.01i}\bigg) ^{101}\bigg]^{i/100}\\ 
&\leo w  \bigg[1600e^{16} \bigg(\frac{4.04}{3} \cdot 0.577\bigg)^{101}\bigg]^{i/100}\\
& \leo w e^{-i/100}.
\end{align*}
Therefore the probability that Case 1 is satisfied is bounded by a constant times
\begin{align*}
& \sum_{w=1}^{10^{-5}n} \sum_{i=10^{-3}n}^{0.5n} w e^{-i/200}=O(n^{-3/4}).
\end{align*}
\vspace{3mm}
\\ \textbf{Case 2:} $r=|C_z|\geq 0.997n$. 
Let $V_1=V(C_z)$, 
$V_2=V\setminus (V_1\cup W)$. First note that $V_2$ spans at least $w$ components. Therefore $|V_2|\geq w$. To lower bound $e(V_2:W)$ we use the following Claim.
\vspace{3mm}
\\ \textbf{Claim 1} Every vertex in $W$ is adjacent to at least 3 distinct components in $V \setminus W$, and hence to at least 2 vertices in $V_2$.
\vspace{3mm}
\\ \textbf{Proof of Claim 1:} Let $v \in W$ be such that it is adjacent to $t\in \{0,1,2\}$ components in $V\setminus W$. Consider $W'=W\setminus\{v\}$. Thus $|W'|=|W|-1$ . If $t=0$ then $q(V\setminus W')=q(V\setminus W)+1$. If $t=1$ then $q(V\setminus W')\geq q(V\setminus W)-1$. If $t=2$ then if the both of the corresponding components have odd size then the new component will also have odd side, while if only one of them has odd size then the new one has even size. Finally if both have even size the new one has odd size. In all three cases the inequality  $q(V\setminus W')\geq q(V\setminus W)-1$ is satisfied. Therefore $q(V\setminus W') \geq q(V\setminus W) -1 >|W|-1=|W'|  $ contradicting the minimality of $W$.
\qed 
\vspace{3mm}
\\ From Claim 1 it follows that $W:V_2$ spans  at least $2w$ edges. 
We also have that $|V_2|\leq n-r-w \leq 0.003n$. For fixed  $2\leq w\leq 10^{-5}n$, $3w\leq d_W \leq 4w$ and $w\leq i$  the probability that there exist such sets $V_1,V_2,W$, $|V_2|=i$, $w=|W|$, $d(W)=d_W$ and $2w\leq  \ell= e(V_2:W)\leq 4w$ is bounded by  
\begin{align*}
&\sum_{\ell=2w}^{d_W} 
 \binom{n}{i} \binom{n-i}{w} \binom{d_W}{\ell} \frac{\phi(d_2+\ell)\cdot \phi(2m-d_2-\ell) }{\phi(2m)}\\
& \leo  \sum_{\ell=2w}^{d_W} \bfrac{en}{i}^i \bfrac{en}{w}^w 2^{4w} \frac{\bfrac{d_2+\ell}{2}^{(d_2+\ell-2w)/2}  \bfrac{2w}{e}^{w} \bfrac{2m-d_2-\ell}{e}^{(2m-d_2-\ell)/2} }{\bfrac{2m}{e}^m}\\ 
&= \sum_{\ell=2w}^{d_W} \bfrac{en}{i}^i \bfrac{en}{w}^w 2^{4w}   \bfrac{2w}{2m}^{w} \brac{\frac{d_2+\ell}{2m}\cdot \frac{e}2}^{(d_2+\ell-2w)/2} 
\bigg(1-\frac{d_2+\ell}{2m}\bigg)^{(2m-d_2-\ell)/2} \\  
& \leo w \bfrac{en}{i}^i \bfrac{16e}{3}^{w} \bigg( \frac{5i}{3n} \cdot \frac{e}{2} \bigg)^{3i/2}\\
&\leo w   \brac{e^2 \bfrac{16e}{3}^{2w/i} \frac{5^3i}{3^3n}}^{i/2}.
\end{align*} 
For the second line we used the second inequality of Lemma \ref{estimates}. For the fourth line we used that $2w\leq \ell $, $d_2+\ell \leq 4|V_2|+4w \leq 0.01204n$ and so $\brac{\frac{d_2+\ell}{2m}\cdot\frac{e}2}^{(d_2+\ell-2w)/2}$ is maximized when $d_2,\ell$ are as small as possible, that is $d_2=3i,\ell=2w$. Furthermore note that $d_2+\ell-2w\geq d_2\geq 3i$ and $i\geq q(V\setminus W)-1\geq w$. Therefore the probability that Case 2 is satisfied is bounded by a constant times
\begin{align*}
&\sum_{w=2}^{10^{-5}n} \sum_{i=w}^{0.003n}  \brac{e^2\bfrac{16e}{3}^{2w/i} \frac{5^3i}{3^3n}}^{i/2}\\
&\leo  \sum_{w=2}^{10^{-5}n} \sum_{i=w}^{2w}\bfrac{C_1i}{n}^{i/2}+\sum_{w=2}^{10^{-5}n} \sum_{i=2w}^{0.003n}\bfrac{C_2i}{n}^{i/2}\\
\noalign{where $C_1=16^25^3e^4/3^5,C_2=16\cdot 5^3 e^3/3^4$,}
&\leq \sum_{i=2}^{n^{1/4}}i\brac{\bfrac{C_1}{n^{3/4}}^{i/2}+\bfrac{C_2}{n^{3/4}}^{i/2}} +\sum_{i=n^{1/4}}^{2\cdot 10^{-5}n}i\bfrac{2C_1}{10^5}^{i/2}+\sum_{i=n^{1/4}}^{0.003n}i\bfrac{6C_2}{10^3}^{i/2}\\
&=O(n^{-3/4}).
\end{align*} 
Finally, since $G$ has an even number of vertices, for $W=\emptyset$ we have $|W|=q(V\setminus W)=0$.
\qed
\subsection{Proof of Lemma \ref{maximal}:}
Let $W$ be a set satisfying $q(V \setminus W)>|W|$ of maximum size and assume $w=|W|\geq 10^{-5}n$.
\vspace{3mm}
\\
{\textbf{Claim 2}} No component induced by $V\setminus W$ is a tree with more than one vertex.
\vspace{3mm}
\\ \textbf{Proof of Claim 2:} Indeed assume that $C_i$ is such a component. If $|C_i|$ is even then
let $v$ be a leaf of $C_i$ and define $W'=W \cup \{v\}$. Then $C_i \setminus \{v\}$ is an odd component in $V \setminus W'$ and $q(V \setminus W')=q(V \setminus W)+1>|W|+1=|W'|$ contradicting the maximality of $W$.

Thus assume that $|C_i|$ is odd. Let $L_1$ be the set of leaves of $C_i$ and $L_2$ be the neighbors of $L_1$. Set $W'=W \cup |L_1|$. Then $|L_1| \geq |L_2|$. Furthermore every vertex in $L_1$ is an odd component in $V \setminus W'$ and in the case $|L_1|=|L_2|$ then $C_i \setminus (L_1\cup L_2)$ is also an odd component in $V \setminus W'$. Therefore,
\begin{align*}
q(V/W') &=q(V/W) -1 +|L_1|+\mathbb{I}(|L_1|=|L_2|) 
\\&\geq q(V/W) +|L_2|  +|L_1|-|L_2| +\mathbb{I}(|L_1|=|L_2|)-1
\\&> |W|+|L_2|=|W'|,
\end{align*}  
contradicting the maximality of $W$. \qed
\vspace{3mm}
\\
We partition $V \setminus W$ into three sets, $W_1,W_2$ and $W_3$, as follows. With the underlying graph being the one spanned by $V \setminus W$, $W_1$ consists of the isolated vertices in $V \setminus W$, $W_2$ consists of the vertices spanned by components that contain a cycle and have  size $s\leq \frac{1}{10}\log n$ and $W_3$ consists of the vertices that are spanned by a component of size at least $\frac1{10}\log n$. Finally let $W_4=W_2 \cup W_3$. To lower bound $W_1$ we use the following  claim.

\textbf{{{Claim 3:}}} W.h.p.\@  $W_4$ spans at most $\frac{11w}{\log n}$ components in $V\setminus W$.

\textbf{Proof of Claim 3:} First observe that the number of components spanned by $W_2$ is smaller than the number of cycles of size at most $\frac1{10} \log n$ in $G$, which we denote by $r$. 
\begin{align*}
\pr(r\geq n^{0.3}) &\leq n^{-0.3} \sum_{q=1}^{0.1 \log n} \binom{n}{q} 4^q q! \frac{\phi(2q) \phi(2m-2q)}{\phi(2m)}\\
& \leo  n^{-0.3} \sum_{q=1}^{0.1 \log n} \bfrac{en}{q}^q 4^q \bfrac{2q}{e}^q \bfrac{e}{2m}^q
\\& \leo n^{-0.3} \sum_{q=1}^{0.1 \log n} \bfrac{8e}{3}^q 
\leo   n^{-0.3} (\log n) 8^{0.1 \log n}=o(1). 
\end{align*}
Hence w.h.p.\@ $W_2$ spans at most $n^{0.3}$ components. Moreover every component spanned by $W_3$ has size at least $\frac{1}{10}\log n$. Therefore $W_4$ spans  at most $n^{0.3}+\frac{10w}{\log n}= \frac{(1+o(1))10w}{\log n}$ components in $V\setminus W$.\qed

Since $W_4$ spans at most $u=\frac{11w}{\log n}$ components in $V\setminus W$ and no component is a tree it follows that the rest of the components consist of at least $q(V\setminus W) -u > w-u$  isolated vertices that lie in $W_1$.

For convenience, we move $|W_1|-(w-u)$ vertices from $W_1$ to $W_4$. Therefore  $|W_1|= w-u$.
Let $k_1$ be the number of vertices of degree 4 in $W_1$ and $d=d(W)-d(W_1)$.
Then $0\leq d\leq 4w-(3(w-u)+k_1)=w+3u-k_1$. For fixed $10^{-5}n\leq w\leq 0.5n$ the probability that there exist such sets $W,W_1,W_4$ is bounded by 

\begin{align}
p_2&\leq \sum_{k_1=0}^{w-u} \sum_{d=0}^{w+3u-k_1}
\binom{n}{2w} \binom{2w}{w}\binom{w}{u} \binom{4w}{d} \pr(d(W)-d(W_1)=d) \label{2}\\ 
&\times  (3(w-u)+k_1)! \times \frac{ [2m- [6(w-u)+2k_1]]!}{2^{m-[3(w-u)+k_1}[m-[3(w-u)+k_1]]!}  \times \frac{2^m m!}{(2m)!}.\label{3}
\end{align}
\textbf{Explanation:} We first choose the sets $W,W_1$  and $W_4$ of size $w,w-u$ and $n-2w+u$ respectively. This can be done in $\binom{n}{2w} \binom{2w}{w}\binom{w}{u} \binom{n-2w+u}{u}^{-1} $ ways, but we ignore the final factor. 

From the at most $4w$ copies of vertices in $W$ we choose a set $W''\subset W$ be of size $d$. 
We let $W'=W/W''$. These are the copies of vertices that will be matched with those in $W_1$.  

In the calculations that follow we let $a=w/n\geq 10^{-5}$. We also let $k_4$  be the number of vertices of degree 4 that lie in $W_4$. We first bound the binomial coefficients, found in the first line.
\begin{align}
  &\binom{n}{2w} \binom{2w}{w}\binom{w}{u} \binom{4w}{d} 
= \binom{n}{2an} \binom{2an}{an}\binom{an}{u} \binom{4an}{d} \nonumber\\
&\leq 2^{o(n)} \bfrac{1}{2a}^{2an} \bfrac{1}{1-2a}^{(1-2a)n} 2^{2an} \bfrac{4ean}{d}^{d} \nonumber\\
& =2^{o(n)}\bfrac{1}{a}^{2an} \bfrac{1}{1-2a}^{(1-2a)n} \bfrac{4ean}{d}^{d}.\label{f1}
\end{align}
For the second line we used that $u\leq u_0$ which implies that $\binom{an}{u}=2^{o(n)}$. Observe that 
\begin{align}\label{2m2}
2m=6(w-u)+2k_1+d+3(n-2w+u)+k_4=3n+d+2k_1+k_4-3u.
\end{align}
Let $m_0=d+2k_1+k_4-3u$. For the terms in line \eqref{3} we have
\begin{align*}
\frac{(2m)!}{2^m m!}&=\frac{(3n)!}{2^{1.5n}(1.5n)!}   
\frac{\prod_{i=1}^{m_0}(3n+i)}{2^{m_0/2} \prod_{i=1}^{m_0/2}(1.5n+i)}
\geo\bfrac{3n}{e}^{1.5n} \prod_{i=1}^{m_0/2}[3n+(2i-1)].\\
& \geq \bfrac{3n}{e}^{1.5n} e^{-o(n)}(3n)^{-3u/2} \prod_{i=1}^{d/2+k_1+k_4/2}[3n+(2i-1)]
\end{align*}
Equation \eqref{2m2} implies that 
$$2m-[6(w-u)+2k_1]=3(1-2a)n+3u+k_4+d.$$ 
Thus,  
\begin{align*}
& \frac{ [2m- [6(w-u)+2k_1]]!}{2^{m-[3(w-u)+2k_1]}[m-[3(w-u)+k_1]]!}=\frac{ [3(1-2a)n]!} {2^{1.5(1-2a)n}[1.5(1-2a)n]!} \cdot \frac{\prod_{i=1}^{d}3(1-2a)n+i}{2^{d/2}\prod_{j=1}^{\frac{d}{2}} 1.5(1-2a)n+j} 
\\&\hspace{5mm} \times \frac{\prod_{i=1}^{k_4}3(1-2a)n+d+i}{2^{k_4/2}\prod_{j=1}^{k_4/2} 1.5(1-2a)n+d/2   +j } \cdot \frac{\prod_{i=1}^{3u}3(1-2a)n+d+k_4+i}{2^{3u/2}\prod_{j=1}^{\frac{3u}{2}} 1.5(1-2a)n+d/2+k_4/2 +j}
\\ &\hspace{5mm}\leo  \bfrac{3(1-2a)n}{e}^{1.5(1-2a)n} \prod_{i=1}^{d/2} [3(1-2a)n+(2i-1)]
 \prod_{j=1}^{k_4/2} [3(1-2a)n+d+(2j-1) ] \cdot (2m)^{3u/2}
\\ &\hspace{5mm} \leo  \bfrac{3(1-2a)n}{e}^{1.5(1-2a)n} [3(1-2a)n+an/2]^{d/2} (2m)^{3u/2} \prod_{j=1}^{k_4/2} [3(1-2a)n+d+(2j-1) ]
\end{align*}
For the last inequality we used the Arithmetic Mean-Geometric Mean inequality and the fact that $d/2\leq an/2+ o(n)$, which follows from $d\leq w+3u-k_1$. 

For the first term of \eqref{3} we have
\begin{align*}
[3(w-u)+k_1]! \leq \frac{3w!}{ (3(w-u))^{3u}} \prod_{i=1}^{k_1}(3(w-u)+i)\leq \bfrac{3an}{e}^{3an} \frac{2^{o(n)}}{n^{3u}} \prod_{i=1}^{k_1}(3(w-u)+i).
\end{align*}
Thus the expression in \eqref{3} is bounded by
\begin{align}
&2^{o(n)} \bfrac{3an}{e}^{3an} \frac{1}{n^{3u}} \prod_{i=1}^{k_1}(3(w-u)+i) \nonumber\\ 
&\times \bfrac{3(1-2a)n}{e}^{1.5(1-2a)n} [3(1-2a)n+an/2]^{d/2} (2m)^{3u/2} \prod_{j=1}^{k_4/2} [3(1-2a)n+d+2j-1 ]\nonumber\\ 
&\times \bigg[ \bfrac{3n}{e}^{1.5n} (3n)^{-3u/2} \prod_{i=1}^{d/2+k_1+k_4/2}[3n+(2i-1)] \bigg]^{-1}\nonumber\\
&=2^{o(n)}  a^{3an}  [(1-2a)n]^{1.5(1-2a)n} \bfrac{6m}{n}^{3u/2}    \prod_{i=1}^{d/2}\frac{ 3(1-2a)n+an/2}{3n+(2i-1)}        \nonumber\\
&\times 
\prod_{i=1}^{k_1}\frac{ 3(w-u)+i}{3n+d+(2i-1)} 
\prod_{i=1}^{k_4/2} \frac{ 3(1-2a)n+d+2i-1}{3n+d+2k_1+2i-1}\nonumber\\
&\leo 2^{o(n)} a^{3an}  [(1-2a)n]^{1.5(1-2a)n}
\prod_{i=1}^{d/2}\frac{ 3(1-2a)n+an/2}{3n} 
\prod_{i=1}^{k_1}\frac{ 3(w-u)+i}{3n+d+(2i-1)} \prod_{i=1}^{k_4/2} 1 \nonumber\\ 
&\leo 2^{o(n)} a^{3an}  [(1-2a)n]^{1.5(1-2a)n}[(1-2a)+a/6]^{d/2} 2^{-k_1} \label{f2}
\end{align}
Finally we consider the term $\pr(d(W)-d(W_1)=d) $ and assume that $h$ vertices of degree 4 were chosen to be included in $W\cup W_1$, so that $d=h+3u-2k_1$. Then, because there are $\binom{h}{k_1}\binom{2w-u-h}{(w-u)-k_1}$ out of $\binom{2w}{w-u}$ ways to  distribute the $k_1$ vertices of degree 4,
\begin{align*}
p_3&=\pr(d(W)-d(W_1)=d) =\binom{h}{k_1}\binom{2w-u-h}{(w-u)-k_1} \bigg/ \binom{2w-u}{w-u}\\
&\leq \binom{h}{k_1}\binom{2w-u-h}{w-u}\bigg/\binom{2w-u}{w-u} \leq \binom{h}{k_1} \prod_{i=0}^{h-1} \frac{w-i}{2w-i}\\
&\leq  2^{hH(k_1/h) -h}=2^{k_1}2^{-k_1+h\cdot H(k_1/h)-h}.
\end{align*}
Here $H(x)=-x\log_2( x) -(1-x) \log_2(1-x)$ is the entropy function. For fixed $d$ we have $h=d+2k_1+o(n)$. Thus 
$$p_3 \leq 2^{o(n)+k_1+df(k_1/d)},\text{ where }f(x)= -x + (1+2x) H\brac{\frac{x}{1+2x}}-(1+2x).$$ 
$f(x)$ has a unique maximum at $x^*$, the solution to $8x(1+x)=(1+2x)^2$ and $f(x^*) \leq -0.771$. Hence 
\beq{f3}{
p_3\leq 2^{-0.771d+k_1+o(n)}.
} 
Multiplying the bounds in \eqref{f1}, \eqref{f2}, \eqref{f3} together we have a bound
\begin{align*}
p_2 & \leq 2^{o(n)-0.771d+k_1} \bfrac{1}{a}^{2an} \bfrac{1}{1-2a}^{(1-2a)n}  \bfrac{4ean}{d}^{d} 
\\ &\times a^{3an} (1-2a)^{1.5(1-2a)n} \bigg(1-2a+ \frac{a}{6}\bigg)^{d/2} 2^{-k_1}
\\ & = 2^{o(n)}  \bfrac{2^{1.229}ean}{d}^{d}  a^{an} (1-2a)^{0.5(1-2a)n} \bigg(1-\frac{11a}{6}\bigg)^{d/2}
\end{align*}
Thus $p_2=o(1)$ when $d=o(n)$. Let $d=ban$ for some $0<b\leq 1$. Then,
\begin{align*}
p_2 & \leq  \bigg\{ 2^{o(1)} \bigg[ \frac{2^{1.229}e}{b}\bigg(1-\frac{11a}{6} \bigg)^{0.5} \bigg]^b  a (1-2a)^{0.5(1-2a)/a}  \bigg\}^{an}
\end{align*}
Let $g(a)= {2^{1.229}e}\bigg(1-\frac{11a}{6} \bigg)^{0.5}$. When $g(a)<e$ then $\bfrac{g(a)}{b}^b$ is maximized when $b=\frac{g(a)}{e}$ which yields
$$ p_2 \leq  \bigg\{ 2^{o(1)} \ e^{ 2^{1.229} (1-\frac{11a}{6})^{0.5}}  a (1-2a)^{0.5(1-2a)/a}  \bigg\}^{an} 
\leq\bfrac{99}{100}^{an}.$$
The last inequality is most easily verified numerically.

When $g(a)>e$ then $\bfrac{g(a)}{b}^b$ is maximized at $b=1$. Hence
$$ p_2 \leq  \bigg\{ 2^{o(1)} \  2^{1.229} e \brac{1-\frac{11a}{6} }^{0.5}   a (1-2a)^{0.5(1-2a)/a}  \bigg\}^{an} \leq \bfrac{19}{20}^{an}.$$
The last inequality is most easily verified numerically. Thus the probability that there exists a set $W$ satisfying $q(V \setminus W)>|W|$ of size $w=|W|\leq 10^{-5}n$
is bounded by 
\begin{align*}
\sum_{w=10^{-5}n}^{0.5n}  \bfrac{99}{100}^w =o(1). 
\end{align*}
This only leaves the case of $n$ odd. The reader will notice that in none of the calculations above, did we use the fact that $n$ was even. The Tutte-Berge formula for the maximum size of a matching $\nu(G)$ is
$$\nu(G)=\min_{W\subseteq V}\frac12(|V|+|W|-q(V\setminus W)).$$
We have shown that the above expression is at least $|V|/2$ for $W\neq\emptyset$ and so the case of $n$ odd is handled by putting $W=\emptyset$ and $q(W)=1$.
\section{Conclusions and open questions}
The paper of Karp and Sipser \cite{KS} has been a springboard for research on matching algorithms in random graphs. Algorithm 1 of that paper has not been the subject of a great deal of analysis, mainly because of the way it disturbs the degree sequences of the graphs that it produces along the way. In this paper we have shown that if the original graph has small maximum degree then the maximum degree is controllable and the great efficiency of the algorithm can be verified.

It is natural to try to extend the analysis to random regular graphs with degree more than four and we are in the process of trying to overcome some technical problems. It would also be of interest to analyse the algorithm on $G_{n,p}$, as originally intended.

\section{Diagrams of Hyperactions of interest}
{\bf Type 2.}

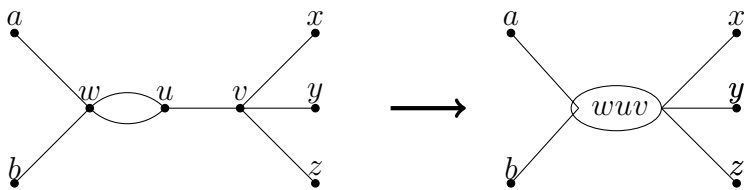
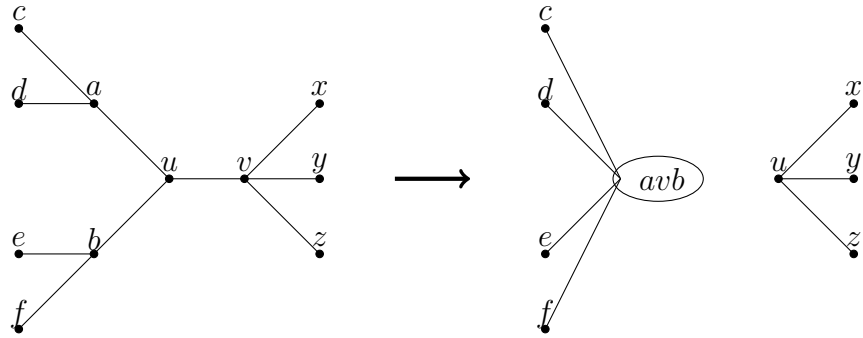
\begin{center}
\pic{
\node at (0,0.2) {$w$};
\node at (1,0.2) {$u$};
\node at (2,0.2) {$v$};
\node at (3,1.2) {$x$};
\node at (3,0.2) {$y$};
\node at (3,-.8) {$z$};
\node at (-1,1.2) {$a$};
\node at (-1,-0.8) {$b$};
\draw (1,0) -- (2,0);
\draw (2,0) -- (3,1);
\draw (2,0) -- (3,-1);
\draw (2,0) -- (3,0);
\draw (0,0) -- (-1,1);
\draw (0,0) -- (-1,-1);
\draw (0,0) to [out=45,in=135] (1,0);
\draw (0,0) to [out=-45,in=-135] (1,0);
\draw [fill=black] (0,0) circle [radius=.05];
\draw [fill=black] (1,0) circle [radius=.05];
\draw [fill=black] (2,0) circle [radius=.05];
\draw [fill=black] (3,0) circle [radius=.05];
\draw [fill=black] (3,1) circle [radius=.05];
\draw [fill=black] (3,-1) circle [radius=.05];
\draw [fill=black] (-1,1) circle [radius=.05];
\draw [fill=black] (-1,-1) circle [radius=.05];
\draw [->] [ultra thick] (4,0) -- (5,0);
\draw (7,0) ellipse (.6 and .3);
\node at (7.05,0) {$wuv$};
\draw (7.6,0) -- (8.6,1);
\draw (7.6,0) -- (8.6,-1);
\draw (7.6,0) -- (8.6,0);
\draw (5.6,1) -- (6.5,0);
\draw (5.6,-1) -- (6.5,0);
\draw [fill=black] (8.6,1) circle [radius=.05];
\draw [fill=black] (8.6,-1) circle [radius=.05];
\draw [fill=black] (8.6,0) circle [radius=.05];
\draw [fill=black] (5.6,-1) circle [radius=.05];
\draw [fill=black] (5.6,1) circle [radius=.05];
\node at (5.6,1.2) {$a$};
\node at (5.6,-0.8) {$b$};
\node at (8.6,1.2) {$x$};
\node at (8.6,0.2) {$y$};
\node at (8.6,-.8) {$z$};
\node at (8.6,0.2) {$y$};
\node at (8.6,-.8) {$z$};
}
\end{center}

{\bf Type 3.}

\begin{center}
\pic{
\node at (1,0.2) {$u$};
\node at (2,0.2) {$v$};
\node at (3,1.2) {$x$};
\node at (3,0.2) {$y$};
\node at (3,-0.8) {$z$};
\node at (0,1.2) {$a$};
\node at (0,-.8) {$b$};
\node at (-1,2.2) {$c$};
\node at (-1,1.2) {$d$};
\node at (-1,-.8) {$e$};
\node at (-1,-1.8) {$f$};
\draw (1,0) -- (2,0);
\draw (2,0) -- (3,1);
\draw (2,0) -- (3,-1);
\draw (2,0) -- (3,0);
\draw (0,1) -- (1,0);
\draw (0,1) -- (-1,2);
\draw (0,1) -- (-1,1);
\draw (0,-1) -- (-1,-1);
\draw (0,-1) -- (-1,-2);
\draw (0,-1) -- (1,0);
\draw [fill=black] (-1,2) circle [radius=.05];
\draw [fill=black] (-1,1) circle [radius=.05];
\draw [fill=black] (-1,-2) circle [radius=.05];
\draw [fill=black] (-1,-1) circle [radius=.05];
\draw [fill=black] (0,1) circle [radius=.05];
\draw [fill=black] (0,-1) circle [radius=.05];
\draw [fill=black] (1,0) circle [radius=.05];
\draw [fill=black] (2,0) circle [radius=.05];
\draw [fill=black] (3,0) circle [radius=.05];
\draw [fill=black] (3,1) circle [radius=.05];
\draw [fill=black] (3,-1) circle [radius=.05];
\draw [->] [ultra thick] (4,0) -- (5,0);
\draw (7.5,0) ellipse (.6 and .3);
\node at (7.55,0) {$avb$};
\draw (9.1,0) -- (10.1,1);
\draw (9.1,0) -- (10.1,-1);
\draw (9.1,0) -- (10.1,0);
\draw [fill=black] (9.1,0) circle [radius=.05];
\draw [fill=black] (10.1,1) circle [radius=.05];
\draw [fill=black] (10.1,-1) circle [radius=.05];
\draw [fill=black] (10.1,0) circle [radius=.05];
\node at (9.1,0.2) {$u$};
\node at (10.1,1.2) {$x$};
\node at (10.1,0.2) {$y$};
\node at (10.1,-.8) {$z$};
\node at (6,2.2) {$c$};
\node at (6,1.2) {$d$};
\node at (6,-.8) {$e$};
\node at (6,-1.8) {$f$};
\draw [fill=black] (6,2) circle [radius=.05];
\draw [fill=black] (6,1) circle [radius=.05];
\draw [fill=black] (6,-1) circle [radius=.05];
\draw [fill=black] (6,-2) circle [radius=.05];
\draw (6,2) -- (7,0);
\draw (6,1) -- (7,0);
\draw (6,-2) -- (7,0);
\draw (6,-1) -- (7,0);
}
\end{center}
We allow the edge $\set{a,b}$ to be a single edge in this construction. This gives us a Type 3b hyperaction.

{\bf Type 4.}

\begin{center}
\pic{
\node at (0,0.2) {$v$};
\draw [fill=black] (0,0) circle [radius=.05];
\node at (-1,1.2) {$a$};
\node at (-1,0.2) {$b$};
\node at (-1,-0.8) {$c$};
\draw [fill=black] (-1,1) circle [radius=.05];
\draw [fill=black] (-1,0) circle [radius=.05];
\draw [fill=black] (-1,-1) circle [radius=.05];
\draw (-1,1) -- (0,0);
\draw (-1,0) -- (0,0);
\draw (-1,-1) -- (0,0);
\draw [fill=black] (1,0) circle [radius=.05];
\node at (1,0.2) {$u$};
\draw (0,0) -- (1,0);
\node at (2,1.2) {$x_1$};
\node at (2,-1.2) {$x_2$};
\draw [fill=black] (2,1) circle [radius=.05];
\draw [fill=black] (2,-1) circle [radius=.05];
\draw (1,0) -- (2,1);
\draw (1,0) -- (2,-1);
\draw (2,-1) -- (2,1);
\node at (3,1.2) {$w_1$};
\node at (3,-1.2) {$w_2$};
\draw [fill=black] (3,1) circle [radius=.05];
\draw [fill=black] (3,-1) circle [radius=.05];
\draw (2,1) -- (3,1);
\draw (2,-1) -- (3,-1);
\node at (4,2.2) {$p$};
\node at (4,1.2) {$q$};
\node at (4,-0.8) {$r$};
\node at (4,-1.8) {$s$};
\draw [fill=black] (4,2) circle [radius=.05];
\draw [fill=black] (4,1) circle [radius=.05];
\draw [fill=black] (4,-1) circle [radius=.05];
\draw [fill=black] (4,-2) circle [radius=.05];
\draw (3,1) -- (4,2);
\draw (3,1) -- (4,1);
\draw (3,-1) -- (4,-2);
\draw (3,-1) -- (4,-1);
\draw [->] [ultra thick] (5,0) -- (6,0);
\node at (8,0.2) {$v$};
\draw [fill=black] (0,0) circle [radius=.05];
\node at (7,1.2) {$a$};
\node at (7,0.2) {$b$};
\node at (7,-0.8) {$c$};
\draw [fill=black] (7,1) circle [radius=.05];
\draw [fill=black] (7,0) circle [radius=.05];
\draw [fill=black] (7,-1) circle [radius=.05];
\draw [fill=black] (8,0) circle [radius=.05];
\draw (8,0) -- (7,1);
\draw (8,0) -- (7,0);
\draw (8,0) -- (7,-1);
\draw (10,0) ellipse (1.4 and .3);
\node at (10,0) {$u,x_1,x_2,w_1,w_2$};
\node at (13,2.2) {$p$};
\node at (13,1.2) {$q$};
\node at (13,-0.8) {$r$};
\node at (13,-1.8) {$s$};
\draw [fill=black] (13,2) circle [radius=.05];
\draw [fill=black] (13,1) circle [radius=.05];
\draw [fill=black] (13,-1) circle [radius=.05];
\draw [fill=black] (13,-2) circle [radius=.05];
\draw (11.3,0) -- (13,2);
\draw (11.3,0) -- (13,1);
\draw (11.3,0) -- (13,-1);
\draw (11.3,0) -- (13,-2);
}
\end{center}

{\bf Type 5}.

\begin{center}
\pic{
\draw [fill=black] (0,1) circle [radius=.05];
\draw [fill=black] (0,-1) circle [radius=.05];
\node at (0,1.2) {$a$};
\node at (0,-.8) {$b$};
\draw [fill=black] (1,0) circle [radius=.05];
\node at (1,.2) {$z$};
\draw (0,1) -- (1,0);
\draw (0,-1) -- (1,0);
\draw [fill=black] (2,0) circle [radius=.05];
\draw (1,0) -- (2,0);
\node at (2,.2) {$u$};
\draw [fill=black] (3,0) circle [radius=.05];
\draw (3,0) -- (2,0);
\node at (3,-.2) {$v$};
\draw [fill=black] (3,1) circle [radius=.05];
\node at (3,1.2) {$x_1$};
\draw (3,1) -- (2,0);
\draw (3,1) -- (3,0);
\draw [fill=black] (4,0) circle [radius=.05];
\node at (4,.2) {$x_2$};
\draw (3,0) -- (4,0);
\draw [fill=black] (5,0) circle [radius=.05];
\draw [fill=black] (5,1) circle [radius=.05];
\draw [fill=black] (5,-1) circle [radius=.05];
\draw (4,0) -- (5,0);
\draw (3,1) -- (5,1);
\draw (4,0) -- (5,-1);
\node at (5,1.2) {$p$};
\node at (5,0.2) {$q$};
\node at (5,-.8) {$r$};
\draw [->] [ultra thick] (6,0) -- (7,0);
\draw (10,0) ellipse (1.4 and .3);
\node at (10,0) {$z,u,v,x_1,x_2$};
\draw [fill=black] (8,1) circle [radius=.05];
\draw [fill=black] (8,-1) circle [radius=.05];
\node at (8,-0.8) {$b$};
\node at (8,1.2) {$a$};
\draw [fill=black] (12.5,0) circle [radius=.05];
\draw [fill=black] (12.5,1) circle [radius=.05];
\draw [fill=black] (12.5,-1) circle [radius=.05];
\draw (11.4,0) -- (12.5,1);
\draw (11.4,0) -- (12.5,-1);
\draw (11.4,0) -- (12.5,0);
\node at (12.5,1.2) {$p$};
\node at (12.5,-0.8) {$r$};
\node at (12.5,0.2) {$q$};
\draw (8,1) -- (8.6,0);
\draw (8,-1) -- (8.6,0);
}
\end{center}

{\bf Type 33.}
\begin{center}
\pic{
\draw [fill=black] (0,1) circle [radius=.05];
\draw [fill=black] (0,-1) circle [radius=.05];
\node at (0,1.2) {$u_1$};
\node at (0,-.8) {$u_2$};
\draw [fill=black] (1,0) circle [radius=.05];
\node at (1,.2) {$u$};
\node at (2,.2) {$v$};
\draw (0,1) -- (1,0);
\draw (0,-1) -- (1,0);
\node at (-1,2.2) {$a$};
\node at (-1,1.2) {$b$};
\node at (-1,-.8) {$c$};
\node at (-1,-1.8) {$d$};
\draw (1,0) -- (2,0);
\draw (2,0) -- (3,1);
\draw (2,0) -- (3,-1);
\draw (0,1) -- (1,0);
\draw (0,1) -- (-1,2);
\draw (0,1) -- (-1,1);
\draw (0,-1) -- (-1,-1);
\draw (0,-1) -- (-1,-2);
\draw (0,-1) -- (1,0);
\draw [fill=black] (-1,2) circle [radius=.05];
\draw [fill=black] (-1,1) circle [radius=.05];
\draw [fill=black] (-1,-2) circle [radius=.05];
\draw [fill=black] (-1,-1) circle [radius=.05];
\draw [fill=black] (0,1) circle [radius=.05];
\draw [fill=black] (0,-1) circle [radius=.05];
\draw [fill=black] (3,1) circle [radius=.05];
\draw [fill=black] (3,-1) circle [radius=.05];
\node at (3,-.8) {$v_2$};
\node at (3,1.2) {$v_2$};
\draw [fill=black] (4,1) circle [radius=.05];
\draw [fill=black] (4,-1) circle [radius=.05];
\node at (4,2.2) {$p$};
\node at (4,1.2) {$q$};
\node at (4,-.8) {$r$};
\node at (4,-1.8) {$s$};
\draw [fill=black] (4,2) circle [radius=.05];
\draw [fill=black] (4,-2) circle [radius=.05];
\draw (3,1) -- (4,1);
\draw (3,1) -- (4,2);
\draw (3,-1) -- (4,-1);
\draw (3,-1) -- (4,-2);
\draw [->] [ultra thick] (5,0) -- (6,0);
\draw (9,0) ellipse (1 and .2);
\draw (12,0) ellipse (1 and .2);
\node at (14,2.2) {$p$};
\node at (14,1.2) {$q$};
\node at (14,-.8) {$r$};
\node at (14,-1.8) {$s$};
\node at (7,2.2) {$a$};
\node at (7,1.2) {$b$};
\node at (7,-.8) {$c$};
\node at (7,-1.8) {$d$};
\draw [fill=black] (14,2) circle [radius=.05];
\draw [fill=black] (14,1) circle [radius=.05];
\draw [fill=black] (14,-2) circle [radius=.05];
\draw [fill=black] (14,-1) circle [radius=.05];
\draw [fill=black] (7,1) circle [radius=.05];
\draw [fill=black] (7,-1) circle [radius=.05];
\draw [fill=black] (7,2) circle [radius=.05];
\draw [fill=black] (7,-2) circle [radius=.05];
\draw (14,2) -- (13,0);
\draw (14,1) -- (13,0);
\draw (14,-1) -- (13,0);
\draw (14,-2) -- (13,0);
\draw (7,2) -- (8,0);
\draw (7,1) -- (8,0);
\draw (7,-2) -- (8,0);
\draw (7,-1) -- (8,0);
\node at (9,0) {$u,u_1,u_2$};
\node at (12,0) {$v,v_1,v_2$};
}
\end{center}

{\bf Type 34.}

\hspace{1in}
\pic{
\draw [fill=black] (0,1) circle [radius=.05];
\draw [fill=black] (0,-1) circle [radius=.05];
\node at (0,1.2) {$v_1$};
\node at (0,-.8) {$v_2$};
\draw [fill=black] (1,0) circle [radius=.05];
\node at (1,.2) {$v$};
\node at (2,.2) {$u$};
\draw (0,1) -- (1,0);
\draw (0,-1) -- (1,0);
\node at (-1,2.2) {$a$};
\node at (-1,1.2) {$b$};
\node at (-1,-.8) {$c$};
\node at (-1,-1.8) {$d$};
\draw (1,0) -- (2,0);
\draw (2,0) -- (3,1);
\draw (2,0) -- (3,-1);
\draw (0,1) -- (1,0);
\draw (0,1) -- (-1,2);
\draw (0,1) -- (-1,1);
\draw (0,-1) -- (-1,-1);
\draw (0,-1) -- (-1,-2);
\draw (0,-1) -- (1,0);
\draw [fill=black] (-1,2) circle [radius=.05];
\draw [fill=black] (-1,1) circle [radius=.05];
\draw [fill=black] (-1,-2) circle [radius=.05];
\draw [fill=black] (-1,-1) circle [radius=.05];
\draw [fill=black] (0,1) circle [radius=.05];
\draw [fill=black] (0,-1) circle [radius=.05];
\draw [fill=black] (3,1) circle [radius=.05];
\draw [fill=black] (3,-1) circle [radius=.05];
\node at (3,-1.2) {$u_2$};
\node at (3,1.2) {$u_1$};
\draw (3,1) -- (3,-1);
\draw [fill=black] (4,1) circle [radius=.05];
\draw [fill=black] (4,-1) circle [radius=.05];
\node at (4,-1.2) {$w_2$};
\node at (4,1.2) {$w_1$};
\draw (3,1) -- (4,1);
\draw (3,-1) -- (4,-1);
\node at (5,2.2) {$p$};
\node at (5,1.2) {$q$};
\node at (5,-.8) {$r$};
\node at (5,-1.8) {$s$};
\draw [fill=black] (5,2) circle [radius=.05];
\draw [fill=black] (5,1) circle [radius=.05];
\draw [fill=black] (5,-2) circle [radius=.05];
\draw [fill=black] (5,-1) circle [radius=.05];
\draw (4,1) -- (5,1);
\draw (4,1) -- (5,2);
\draw (4,-1) -- (5,-1);
\draw (4,-1) -- (5,-2);
\draw [->] [ultra thick] (6,0) -- (7,0);
}

\hspace{3in}
\pic{
\draw (10,0) ellipse (1 and .2);
\draw (13,0) ellipse (1.4 and .3);
\node at (10,0) {$v,v_1,v_2$};
\node at (13,0) {$u,u_1,u_2,w_1,w_2$};
\node at (8,2.2) {$a$};
\node at (8,1.2) {$b$};
\node at (8,-.8) {$c$};
\node at (8,-1.8) {$d$};
\draw [fill=black] (8,2) circle [radius=.05];
\draw [fill=black] (8,1) circle [radius=.05];
\draw [fill=black] (8,-2) circle [radius=.05];
\draw [fill=black] (8,-1) circle [radius=.05];
\draw (9,0) -- (8,1);
\draw (9,0) -- (8,2);
\draw (9,0) -- (8,-1);
\draw (9,0) -- (8,-2);
\node at (15.4,2.2) {$p$};
\node at (15.4,1.2) {$q$};
\node at (15.4,-.8) {$r$};
\node at (15.4,-1.8) {$s$};
\draw [fill=black] (15.4,2) circle [radius=.05];
\draw [fill=black] (15.4,1) circle [radius=.05];
\draw [fill=black] (15.4,-2) circle [radius=.05];
\draw [fill=black] (15.4,-1) circle [radius=.05];
\draw (14.4,0) -- (15.4,1);
\draw (14.4,0) -- (15.4,2);
\draw (14.4,0) -- (15.4,-1);
\draw (14.4,0) -- (15.4,-2);
}

\end{document}